\documentclass{amsart}
\usepackage{amsmath,amssymb,amsthm}
\usepackage{amsaddr}
\usepackage{graphicx}
\usepackage{enumitem}

\usepackage[colorlinks=true,linkcolor=blue,citecolor=blue,pdfpagelabels=false]{hyperref}

\theoremstyle{plain}
  \newtheorem{theorem}{Theorem}[section]
  \newtheorem{lemma}[theorem]{Lemma}
  \newtheorem{conjecture}[theorem]{Conjecture}

\theoremstyle{definition}
  \newtheorem{example}[theorem]{Example}
  \newtheorem{remark}[theorem]{Remark}

\newenvironment{acknowledgements}{\bigskip\textbf{Acknowledgements.}}{}
\newenvironment{enumeratealpha}{\begin{enumerate}[label=(\alph*)]}{\end{enumerate}}

\renewcommand{\geq}{\geqslant}
\renewcommand{\leq}{\leqslant}
\newcommand{\nequiv}{\not\equiv}

\begin{document}

\title[Multivariate Ap\'ery numbers and supercongruences]{Multivariate Ap\'ery numbers and supercongruences of rational functions}

\author{Armin Straub}
\address{Department of Mathematics,
University of Illinois at Urbana-Champaign
\\ \& \\Max-Planck-Institut f\"ur Mathematik}
\email{astraub@illinois.edu}

\begin{abstract}
  One of the many remarkable properties of the Ap\'ery numbers $A (n)$,
  introduced in Ap\'ery's proof of the irrationality of $\zeta (3)$, is that
  they satisfy the two-term supercongruences
  \begin{equation*}
    A (p^r m) \equiv A (p^{r - 1} m) \pmod{p^{3 r}}
  \end{equation*}
  for primes $p \geq 5$. Similar congruences are conjectured to hold for
  all Ap\'ery-like sequences. We provide a fresh perspective on the
  supercongruences satisfied by the Ap\'ery numbers by showing that they
  extend to all Taylor coefficients $A (n_1, n_2, n_3, n_4)$ of the rational
  function
  \begin{equation*}
    \frac{1}{(1 - x_1 - x_2) (1 - x_3 - x_4) - x_1 x_2 x_3 x_4} .
  \end{equation*}
  The Ap\'ery numbers are the diagonal coefficients of this function, which
  is simpler than previously known rational functions with this property.
  
  Our main result offers analogous results for an infinite family of
  sequences, indexed by partitions $\lambda$, which also includes the Franel
  and Yang--Zudilin numbers as well as the Ap\'ery numbers corresponding to
  $\zeta (2)$. Using the example of the Almkvist--Zudilin numbers, we further
  indicate evidence of multivariate supercongruences for other Ap\'ery-like
  sequences.
\end{abstract}



\date{October 22, 2014}

\maketitle

\section{Introduction}

The {\emph{Ap\'ery numbers}}
\begin{equation}
  A (n) = \sum_{k = 0}^n \binom{n}{k}^2 \binom{n + k}{k}^2 \label{eq:apery}
\end{equation}
played a crucial role in R.~Ap\'ery's proof \cite{apery}, \cite{alf} of
the irrationality of $\zeta (3)$ and have inspired much further work. Among
many other interesting properties, they satisfy congruences with surprisingly
large moduli, referred to as {\emph{supercongruences}}, a term coined by
F.~Beukers \cite{beukers-apery85}. For instance, for all primes $p \geq
5$ and all positive integers $r$,
\begin{equation}
  A (p^r m) \equiv A (p^{r - 1} m) \pmod{p^{3 r}} .
  \label{eq:apery-sc}
\end{equation}
The special case $m = 1$, $r = 1$ was conjectured by S.~Chowla, J.~Cowles and
M.~Cowles \cite{ccc-apery}, who established the corresponding congruence
modulo $p^2$. The case $r = 1$ was subsequently shown by I.~Gessel
\cite{gessel-super} and Y.~Mimura \cite{mimura-apery}, while the general
case has been proved by M.~Coster \cite{coster-sc}. The proof is an adaption
of F.~Beukers' proof \cite{beukers-apery85} of the related congruence
\begin{equation}
  A (p^r m - 1) \equiv A (p^{r - 1} m - 1) \pmod{p^{3 r}},
  \label{eq:apery-sc1}
\end{equation}
again valid for all primes $p \geq 5$ and all positive integers $r$. That
congruence \eqref{eq:apery-sc1} can be interpreted as an extension of
\eqref{eq:apery-sc} to negative integers is explained in
Remark~\ref{rk:apery-shifted}. For further congruence properties of the
Ap\'ery numbers we refer to \cite{cowles-apery}, \cite{beukers-apery87},
\cite{ahlgren-ono-apery}, \cite{kilbourn-apery06}.

Given a series
\begin{equation}
  F (x_1, \ldots, x_d) = \sum_{n_1, \ldots, n_d \geq 0} a (n_1, \ldots,
  n_d) x_1^{n_1} \cdots x_d^{n_d}, \label{eq:taylor}
\end{equation}
its {\emph{diagonal coefficients}} are the coefficients $a (n, \ldots, n)$ and
the {\emph{diagonal}} is the ordinary generating function of the diagonal
coefficients. For our purposes, $F$ will always be a rational function. It is
well-known, see for instance \cite[Theorem 5.2]{lp-diag}, that the diagonal
of a rational function satisfies a Picard--Fuchs linear differential equation
and as such ``comes from geometry''. In particular, the diagonal coefficients
satisfy a linear recurrence with polynomial coefficients.

Many sequences of number-theoretic interest can be represented as the diagonal
coefficients of rational functions. In particular, it is known
\cite{christol-diag84}, \cite{lp-diag} that the Ap\'ery numbers are the
diagonal coefficients of the rational function
\begin{equation}
  \frac{1}{(1 - x_1) [ (1 - x_2) (1 - x_3) (1 - x_4) (1 - x_5) - x_1 x_2 x_3]}
  . \label{eq:apery-rf5}
\end{equation}
Several other rational functions of which the Ap\'ery numbers are the
diagonal coefficients are given in \cite{bbchm-diag}, where it is also
discussed how these can be obtained from the representation of the Ap\'ery
numbers as the binomial sum \eqref{eq:apery}. However, all of these rational
function involve at least five variables and, in each case, the polynomial in
the denominator factors. Our first result shows that, in fact, the Ap\'ery
numbers are the diagonal coefficients of a simpler rational function in only
four variables.

\begin{theorem}
  \label{thm:apery-diag}The Ap\'ery numbers $A (n)$, defined in
  \eqref{eq:apery}, are the diagonal coefficients of
  \begin{equation}
    \frac{1}{(1 - x_1 - x_2) (1 - x_3 - x_4) - x_1 x_2 x_3 x_4} .
    \label{eq:apery-rf}
  \end{equation}
\end{theorem}

Representing a sequence as the diagonal of a rational function has certain
benefits. For instance, asymptotic results can be obtained directly and
explicitly from this rational function. This is the subject of
{\emph{multivariate asymptotics}}, as developed in \cite{pw-masy1}. For
details and a host of worked examples we refer to \cite{pw-masy-eg}. As a
second example, the rational generating function provides a means to compute
the sequence modulo a fixed prime power. Indeed, the diagonal of a rational
function with integral Taylor coefficients, such as \eqref{eq:apery-rf}, is
algebraic modulo $p^{\alpha}$ for any $\alpha$ \cite{lp-diag}. A recent
demonstration that this can be done very constructively is given in
\cite{ry-diag13}, where the values modulo $p^{\alpha}$ of sequences such as
the Ap\'ery numbers are, equivalently, encoded as finite automata.

We note that a statement such as Theorem~\ref{thm:apery-diag} is more or less
automatic to prove once discovered. For instance, given a rational function,
we can always repeatedly employ a binomial series expansion to represent the
Taylor coefficients as a nested sum of hypergeometric terms. In principle,
creative telescoping \cite{aeqb} will then obtain a linear recurrence
satisfied by the diagonal coefficients, in which case it suffices to check
that the alternative expression satisfies the same recurrence and agrees for
sufficiently many initial values.

For the rational function $F (\boldsymbol{x})$ given in \eqref{eq:apery-rf}, we
can gain considerably more insight. Indeed, for all the Taylor coefficients $A
(\boldsymbol{n})$, defined by
\begin{equation}
  F (x_1, x_2, x_3, x_4) = \sum_{n_1, n_2, n_3, n_4 \geq 0} A (n_1, n_2,
  n_2, n_4) x_1^{n_1} x_2^{n_2} x_3^{n_3} x_4^{n_4}, \label{eq:apery4}
\end{equation}
we find, for instance by applying MacMahon's Master Theorem
\cite{macmahon-comb} as detailed in Section~\ref{sec:coeffs}, the explicit
formula
\begin{equation}
  A (\boldsymbol{n}) = \sum_{k \in \mathbb{Z}} \binom{n_1}{k} \binom{n_3}{k}
  \binom{n_1 + n_2 - k}{n_1} \binom{n_3 + n_4 - k}{n_3}, \label{eq:apery4x}
\end{equation}
of which Theorem~\ref{thm:apery-diag} is an immediate consequence.

An instance of our main result is the observation that the supercongruence
\eqref{eq:apery-sc} for the Ap\'ery numbers generalizes to all coefficients
\eqref{eq:apery4x} of the rational function \eqref{eq:apery-rf} in the
following sense.

\begin{theorem}
  \label{thm:apery4-sc}Let $\boldsymbol{n}= (n_1, n_2, n_3, n_4) \in
  \mathbb{Z}^4$. The coefficients $A (\boldsymbol{n})$, defined in
  \eqref{eq:apery4} and extended to negative integers by \eqref{eq:apery4x},
  satisfy, for primes $p \geq 5$ and positive integers $r$, the
  supercongruences
  \begin{equation}
    A (p^r \boldsymbol{n}) \equiv A (p^{r - 1} \boldsymbol{n}) \pmod{p^{3 r}} . \label{eq:apery-scr}
  \end{equation}
\end{theorem}

Note that the Ap\'ery numbers are $A (n) = A (n, n, n, n)$ so that
\eqref{eq:apery-scr} indeed generalizes \eqref{eq:apery-sc}. Our reason for
allowing negative entries in $\boldsymbol{n}$ is that by doing so, we also
generalize Beukers' supercongruence \eqref{eq:apery-sc1}. Indeed, as explained
in Remark~\ref{rk:apery-shifted} below, $A (n - 1) = A (- n, - n, - n, - n)$.
Theorem~\ref{thm:apery4-sc} is a special case of our main result,
Theorem~\ref{thm:aperyx-sc}, in which we prove such supercongruences for an
infinite family of sequences. This family includes other Ap\'ery-like
sequences such as the Franel and Yang--Zudilin numbers as well as the
Ap\'ery numbers corresponding to $\zeta (2)$.

We therefore review Ap\'ery-like sequences in Section~\ref{sec:aperyx}.
Though no uniform reason is known, each Ap\'ery-like sequence appears to
satisfy a supercongruence of the form \eqref{eq:apery-sc}, some of which have
been proved \cite{beukers-apery85}, \cite{coster-sc}, \cite{ccs-apery},
\cite{os-supercong11}, \cite{os-domb}, \cite{oss-sporadic} while others
remain open \cite{oss-sporadic}. A major motivation for this note is to work
towards an understanding of this observation. Our contribution to this
question is the insight that, at least for several Ap\'ery-like sequences,
these supercongruences generalize to all coefficients of a rational function.
Our main result, which includes the case of the Ap\'ery numbers outlined in
this introduction, is given in Section~\ref{sec:main}. In that section, we
also record two further conjectural instances of this phenomenon. Finally, we
provide proofs for our results in Sections~\ref{sec:coeffs} and \ref{sec:sc}.

\begin{remark}
  \label{rk:apery-shifted}Let us indicate that congruence \eqref{eq:apery-sc1}
  can be interpreted as the natural extension of \eqref{eq:apery-sc} to the
  case of negative integers $m$. To see this, generalize the definition
  \eqref{eq:apery} of the Ap\'ery numbers $A (n)$ to all integers $n$ by
  setting
  \begin{equation}
    A (n) = \sum_{k \in \mathbb{Z}} \binom{n}{k}^2 \binom{n + k}{k}^2 .
    \label{eq:apery:neg}
  \end{equation}
  Here, we assume the values of the binomial coefficients to be defined as the
  (limiting) values of the corresponding quotient of gamma functions, that is,
  \begin{equation*}
    \binom{n}{k} = \lim_{z \rightarrow 0} \frac{\Gamma (z + n + 1)}{\Gamma (z
     + k + 1) \Gamma (z + n - k + 1)} .
  \end{equation*}
  Since $\Gamma (z + 1)$ has no zeros, and poles only at negative integers
  $z$, one observes that the binomial coefficient $\binom{n}{k}$ is finite for
  all integers $n$ and $k$. Moreover, the binomial coefficient with integer
  entries is nonzero only if either $k \geq 0$ and $n - k \geq 0$,
  or if $n < 0$ and $k \geq 0$, or if $n < 0$ and $n - k \geq 0$.
  Note that in each of these cases $k \geq 0$ or $n - k \geq 0$, so
  that the symmetry $\binom{n}{k} = \binom{n}{n - k}$ allows us to compute
  these binomial coefficients in the obvious way. For instance, $\binom{- 3}{-
  5} = \binom{- 3}{2} = \frac{(- 3) (- 4)}{2!} = 6$. As carefully shown in
  \cite{sprugnoli-binom}, for all integers $n$ and $k$, we have the negation
  rule
  \begin{equation}
    \binom{n}{k} = \operatorname{sgn} (k) (- 1)^k \binom{- n + k - 1}{k},
    \label{eq:binom-neg}
  \end{equation}
  where $\operatorname{sgn} (k) = 1$ for $k \geq 0$ and $\operatorname{sgn} (k) = - 1$
  for $k < 0$. Applying \eqref{eq:binom-neg} to the sum \eqref{eq:apery:neg},
  we find that
  \begin{equation*}
    A (- n) = A (n - 1) .
  \end{equation*}
  In particular, the congruence \eqref{eq:apery-sc1} is equivalent to
  \eqref{eq:apery-sc} with $- m$ in place of $m$.
\end{remark}

\begin{remark}
  \label{rk:dwork}The proof of formula \eqref{eq:apery4x} in
  Section~\ref{sec:coeffs} shows that the coefficients can be expressed as
  \begin{equation*}
    A (n_1, n_2, n_3, n_4) = \operatorname{ct} \frac{(x_1 + x_2 + x_3)^{n_1} (x_1 +
     x_2)^{n_2} (x_3 + x_4)^{n_3} (x_2 + x_3 + x_4)^{n_4}}{x_1^{n_1} x_2^{n_2}
     x_3^{n_3} x_4^{n_4}},
  \end{equation*}
  representing them as the constant terms of Laurent polynomials. In
  particular, the Ap\'ery numbers \eqref{eq:apery} are the constant term of
  powers of a Laurent polynomial. Namely,
  \begin{equation*}
    A (n) = \operatorname{ct} \left[ \frac{(x_1 + x_2) (x_3 + 1) (x_1 + x_2 + x_3)
     (x_2 + x_3 + 1)}{x_1 x_2 x_3} \right]^n .
  \end{equation*}
  Since the Newton polyhedron of this Laurent polynomial has the origin as its
  only interior integral point, the results of \cite{sd-laurent09},
  \cite{mv-laurent13} apply to show that $A (n)$ satisfies the Dwork
  congruences
  \begin{equation*}
    A (p^r m + n) A (\lfloor n / p \rfloor) \equiv A (p^{r - 1} m + \lfloor n
     / p \rfloor) A (n) \pmod{p^r}
  \end{equation*}
  for all primes $p$ and all integers $m, n \geq 0$, $r \geq 1$. In
  particular,
  \begin{equation}
    A (p^r m) \equiv A (p^{r - 1} m) \pmod{p^r},
    \label{eq:apery-sc-r1}
  \end{equation}
  which is a weaker version of \eqref{eq:apery-sc} that holds for the large
  class of sequences represented as the constant term of powers of a Laurent
  polynomial, subject only to the condition on the Newton polyhedron. This
  gives another indication why congruence \eqref{eq:apery-sc} is referred to
  as a supercongruence. It would be of considerable interest to find similarly
  well-defined classes of sequences for which supercongruences, of the form
  \eqref{eq:apery-sc-r1} but modulo $p^{k r}$ for $k > 1$, hold. Let us note
  that the case $r = 1$ of the Dwork congruences implies the Lucas congruences
  \begin{equation*}
    A ( n) \equiv A ( n_0) A ( n_1) \cdots A ( n_{\ell}) \pmod{p},
  \end{equation*}
  where $n_0, \ldots, n_{\ell} \in \{ 0, 1, \ldots, p - 1 \}$ are the $p$-adic
  digits of $n = n_0 + n_1 p + \ldots + n_{\ell} p^{\ell}$. It is shown in
  \cite{ry-diag13} that Lucas congruences hold for all Taylor coefficients
  of certain rational functions. Additional divisibility properties in this
  direction are obtained in \cite{delaygue-apery} for Ap\'ery-like numbers
  as well as for constant terms of powers of certain Laurent polynomials.
  Finally, we note that an extension of Dwork congruences to the multivariate
  setting has been considered in \cite{kr-multivariate}. In contrast to our
  approach, where, for instance, the Ap\'ery numbers appear as the diagonal
  (multivariate) Taylor coefficients of a multivariate function $F
  (\boldsymbol{x})$, the theory developed in \cite{kr-multivariate} is
  concerned with functions $G (\boldsymbol{x}) = G (x_1, \ldots, x_d)$ for
  which, say, the Ap\'ery numbers are the (univariate) Taylor coefficients
  of the specialization $G (x, \ldots, x)$.
\end{remark}

\section{Review of Ap\'ery-like numbers}\label{sec:aperyx}

The Ap\'ery numbers $A (n)$ are characterized by the $3$-term recurrence
\begin{equation}
  (n + 1)^3 u_{n + 1} = (2 n + 1) (a n^2 + a n + b) u_n - n (c n^2 + d) u_{n -
  1}, \label{eq:rec3-abcd}
\end{equation}
where $(a, b, c, d) = (17, 5, 1, 0)$, together with the initial conditions
\begin{equation}
  u_{- 1} = 0, \hspace{1em} u_0 = 1. \label{eq:apery-rec0}
\end{equation}
As explained in \cite{beukers-dwork}, the fact, that in the recursion
\eqref{eq:rec3-abcd} we divide by $(n + 1)^3$ at each step, means that we
should expect the denominator of $u_n$ to grow like $(n!)^3$. While this is
what happens for generic choice of the parameters $(a, b, c, d)$, the
Ap\'ery numbers have the, from this perspective, exceptional property of
being integral. Initiated by Beukers \cite{beukers-dwork}, systematic
searches have therefore been conducted for recurrences of this kind, which
share the property of having an integer solution with initial conditions
\eqref{eq:apery-rec0}. This was done by D.~Zagier \cite{zagier4} for
recurrences of the form
\begin{equation}
  (n + 1)^2 u_{n + 1} = (a n^2 + a n + b) u_n - c n^2 u_{n - 1},
  \label{eq:rec2-abc}
\end{equation}
by G.~Almkvist and W.~Zudilin \cite{az-de06} for recurrences of the form
\eqref{eq:rec3-abcd} with $d = 0$ and, more recently, by S.~Cooper
\cite{cooper-sporadic} for recurrences of the form \eqref{eq:rec3-abcd}. In
each case, apart from degenerate cases, only finitely many sequences have been
discovered. For details and a possibly complete list of the sequences, we
refer to \cite{zagier4}, \cite{az-de06}, \cite{asz-clausen},
\cite{cooper-sporadic}.

Remarkably, and still rather mysteriously, all of these sequences, often
referred to as {\emph{Ap\'ery-like}}, share some of the interesting
properties of the Ap\'ery numbers. For instance, they all are the
coefficients of modular forms expanded in terms of a corresponding modular
function. In the case of the Ap\'ery numbers $A (n)$, for instance, it was
shown by Beukers \cite{beukers-apery87} that
\begin{equation}
  \sum_{n \geq 0} A (n)  \left( \frac{\eta (\tau) \eta (6 \tau)}{\eta (2
  \tau) \eta (3 \tau)} \right)^{12 n} = \frac{\eta^7 (2 \tau) \eta^7 (3
  \tau)}{\eta^5 (\tau) \eta^5 (6 \tau)}, \label{eq:apery-mp}
\end{equation}
where $\eta (\tau)$ is the Dedekind eta function $\eta ( \tau) = e^{\pi i \tau
/ 12} \prod_{n \geq 1} (1 - e^{2 \pi i n \tau})$. The modular function
and the modular form appearing in \eqref{eq:apery-mp} are modular with respect
to the congruence subgroup $\Gamma_0 (6)$ of level $6$ (in fact, they are
modular with respect to a slightly larger group). While this relation with
modular forms can be proven in each individual case, no conceptual explanation
is available in the sense that if an additional Ap\'ery-like sequence was
found we would not know \textit{a priori} that its generating function has a
modular parametrization such as \eqref{eq:apery-mp}.

As a second example, it is conjectured and in some cases proven
\cite{oss-sporadic} that each Ap\'ery-like sequence satisfies a
supercongruence of the form \eqref{eq:apery-sc}. Again, no uniform explanation
is available and, the known proofs \cite{gessel-super},
\cite{mimura-apery}, \cite{beukers-apery85}, \cite{coster-sc} of the
supercongruences \eqref{eq:apery-sc} and \eqref{eq:apery-sc1} all rely on the
explicit binomial representation \eqref{eq:apery} of the Ap\'ery numbers.
However, not all Ap\'ery-like sequences have a comparably effective binomial
representation so that, for instance, for the {\emph{Almkvist--Zudilin
numbers}} \cite[sequence (4.12){\hspace{0.25em}}($\delta$)]{asz-clausen},
\cite{cz-apery}, \cite{ccs-apery}
\begin{equation}
  Z (n) = \sum_{k = 0}^n ( - 3)^{n - 3 k} \binom{n}{3 k} \binom{n + k}{n}
  \frac{(3 k) !}{k!^3}, \label{eq:az}
\end{equation}
which solve \eqref{eq:rec3-abcd} with $(a, b, c, d) = (- 7, - 3, 81, 0)$, the
supercongruence
\begin{equation}
  Z (p^r m) \equiv Z (p^{r - 1} m) \pmod{p^{3 r}}
  \label{eq:az-sc}
\end{equation}
for primes $p \geq 3$ is conjectural only.

It would therefore be of particular interest to find alternative approaches to
proving supercongruences. In this paper, we provide a new perspective on
supercongruences of the form \eqref{eq:az-sc} by showing that they hold, at
least for several Ap\'ery-like sequences, for all coefficients $C (
\boldsymbol{n})$ of a corresponding rational function, which has the sequence of
interest as its diagonal coefficients. In such a case, one may then hope to
use properties of the rational function to prove, for some $k > 1$, the
supercongruence $C ( p^r \boldsymbol{n}) \equiv C ( p^{r - 1} \boldsymbol{n})$
modulo $p^{k r}$. For instance, for fixed $p^r$, these congruences can be
proved, at least in principle, by computing the multivariate generating
functions of both $C ( p^r \boldsymbol{n})$ and $C ( p^{r - 1} \boldsymbol{n})$,
which are rational functions because they are multisections of a rational
function, and comparing them modulo $p^{k r}$.

Let us note that, in Example~\ref{eg:az} below, we give a characterization of
the Almkvist--Zudilin numbers \eqref{eq:az} as the diagonal of a surprisingly
simple rational function and conjecture that the supercongruences
\eqref{eq:az-sc}, which themselves have not been proved yet, again extend to
all coefficients of this rational function. We hope that the simplicity of the
rational function might help inspire a proof of these supercongruences.

\section{Main result and examples}\label{sec:main}

We now generalize what we have illustrated in the introduction for the
Ap\'ery numbers $A (n)$ to an infinite family of sequences $A_{\lambda,
\varepsilon} (n)$, indexed by partitions $\lambda$ and $\varepsilon \in \{ -
1, 1 \}$, which includes other Ap\'ery-like numbers such as the Franel and
Yang--Zudilin numbers as well as the sequence used by Ap\'ery in relation
with $\zeta (2)$. Our main theorem is Theorem~\ref{thm:aperyx-sc}, in which we
prove (multivariate) supercongruences for this family of sequences, thus
unifying and extending a number of known supercongruences. To begin with, the
sequences we are concerned with are introduced by the following extension of
formula \eqref{eq:apery4x}. Here, $\boldsymbol{x}^{\boldsymbol{n}}$ is short for
$x_1^{n_1} x_2^{n_2} \cdots x_d^{n_d}$.

\begin{theorem}
  \label{thm:aperyx-rf}Let $\alpha \in \mathbb{C}$ and $\lambda = (\lambda_1,
  \ldots, \lambda_{\ell}) \in \mathbb{Z}_{> 0}^{\ell}$ with $d = \lambda_1 +
  \ldots + \lambda_{\ell}$, and set $s (j) = \lambda_1 + \ldots + \lambda_{j -
  1}$. Then the Taylor coefficients of the rational function
  \begin{equation}
    \left( \prod_{j = 1}^{\ell} \left[ 1 - \sum_{r = 1}^{\lambda_j} x_{s (j) +
    r} \right] - \alpha x_1 x_2 \cdots x_d \right)^{- 1} = \sum_{\boldsymbol{n}
    \in \mathbb{Z}_{\geq 0}^d} A_{\lambda, \alpha} (\boldsymbol{n})
    \boldsymbol{x}^{\boldsymbol{n}} \label{eq:aperyx-rf}
  \end{equation}
  are given by
  \begin{equation}
    A_{\lambda, \alpha} (\boldsymbol{n}) = \sum_{k \in \mathbb{Z}} \alpha^k
    \prod_{j = 1}^{\ell} \binom{n_{s (j) + 1} + \ldots +_{} n_{s (j) +
    \lambda_j} - (\lambda_j - 1) k}{n_{s (j) + 1} - k, \ldots, n_{s (j) +
    \lambda_j} - k, k} . \label{eq:Alambda}
  \end{equation}
\end{theorem}

The proof of this elementary but crucial result will be given in
Section~\ref{sec:coeffs}. Observe that the multivariate Ap\'ery numbers $A
(\boldsymbol{n})$, defined in \eqref{eq:apery4x}, are the special case $A_{(2,
2), 1} (\boldsymbol{n})$.

Our main result, of which Theorem~\ref{thm:apery4-sc} is the special case
$\lambda = (2, 2)$ and $\varepsilon = 1$, follows next. Note that, if
$\boldsymbol{n} \in \mathbb{Z}_{\geq 0}^d$, then the sum
\eqref{eq:Alambda} defining $A_{\lambda, \alpha} (\boldsymbol{n})$ is finite and
runs over $k = 0, 1, \ldots, \min (n_1, \ldots, n_d)$. On the other hand, if
$\max (\lambda_1, \ldots, \lambda_{\ell}) \geq 2$, then $A_{\lambda,
\alpha} (\boldsymbol{n})$ is finite for any $\boldsymbol{n} \in \mathbb{Z}^d$.

\begin{theorem}
  \label{thm:aperyx-sc}Let $\varepsilon \in \{ - 1, 1 \}$, $\lambda =
  (\lambda_1, \ldots, \lambda_{\ell}) \in \mathbb{Z}_{> 0}^{\ell}$ and assume
  that $\boldsymbol{n} \in \mathbb{Z}^d$, $d = \lambda_1 + \ldots +
  \lambda_{\ell}$, is such that $A_{\lambda, \varepsilon} (\boldsymbol{n})$, as
  defined in \eqref{eq:Alambda}, is finite.
  \begin{enumeratealpha}
    \item \label{thm:aperyx-sc2}If $\ell \geq 2$, then, for all primes $p
    \geq 3$ and integers $r \geq 1$,
    \begin{equation}
      A_{\lambda, \varepsilon} (p^r \boldsymbol{n}) \equiv A_{\lambda,
      \varepsilon} (p^{r - 1} \boldsymbol{n}) \pmod{p^{2 r}}
      . \label{eq:aperyx-sc2}
    \end{equation}
    If $\varepsilon = 1$, then these congruences also hold for $p = 2$.
    
    \item \label{thm:aperyx-sc3}If $\ell \geq 2$ and $\max (\lambda_1,
    \ldots, \lambda_{\ell}) \leq 2$, then, for primes $p \geq 5$ and
    integers $r \geq 1$,
    \begin{equation}
      A_{\lambda, \varepsilon} (p^r \boldsymbol{n}) \equiv A_{\lambda,
      \varepsilon} (p^{r - 1} \boldsymbol{n}) \pmod{p^{3 r}}
      . \label{eq:aperyx-sc}
    \end{equation}
  \end{enumeratealpha}
\end{theorem}

A proof of Theorem~\ref{thm:aperyx-sc} is given in Section~\ref{sec:sc}. One
of the novel features of the proof, which is based on the approach of Gessel
\cite{gessel-super} and Beukers \cite{beukers-apery85}, is that it
proceeds in a uniform fashion for all $\boldsymbol{n} \in \mathbb{Z}^d$. As
outlined in Remark~\ref{rk:apery-shifted}, this allows us to also conclude,
and to a certain extent explain, the shifted supercongruences
\eqref{eq:apery-sc1}, which, among Ap\'ery-like numbers, are special to the
Ap\'ery numbers as well as their version \eqref{eq:aperyB} related to $\zeta
( 2)$. In cases where $\boldsymbol{n}$ has negative entries, the summation
\eqref{eq:Alambda}, while still finite, may include negative values for $k$
(see Remark~\ref{rk:apery-shifted}). We therefore extend classical results,
such as Jacobsthal's binomial congruences, to the case of binomial
coefficients with negative entries.

\begin{example}
  For $\lambda = (2)$, the numbers \eqref{eq:Alambda} specialize to the
  {\emph{Delannoy numbers}}
  \begin{equation*}
    A_{(2), 1} (\boldsymbol{n}) = \sum_{k \in \mathbb{Z}} \binom{n_1}{k}
     \binom{n_1 + n_2 - k}{n_1},
  \end{equation*}
  which, for $n_1, n_2 \geq 0$, count the number of lattice paths from
  $(0, 0)$ to $(n_1, n_2)$ with steps $(1, 0)$, $(0, 1)$ and $(1, 1)$. The
  Delannoy numbers do not satisfy \eqref{eq:aperyx-sc2} or
  \eqref{eq:aperyx-sc}, thus demonstrating the necessity of the condition
  $\ell \geq 2$ in Theorem~\ref{thm:aperyx-sc}. They do satisfy
  \eqref{eq:aperyx-sc2} modulo $p^r$, by virtue of Remark~\ref{rk:dwork}.
\end{example}

\begin{example}
  The Ap\'ery-like sequence
  \begin{equation}
    B (n) = \sum_{k \in \mathbb{Z}} \binom{n}{k}^2 \binom{n + k}{k},
    \label{eq:aperyB}
  \end{equation}
  which satisfies recurrence \eqref{eq:rec2-abc} with $(a, b, c) = (11, 3, -
  1)$, was introduced by Ap\'ery \cite{apery}, \cite{alf} along with
  \eqref{eq:apery} and used to (re)prove the irrationality of $\zeta (2)$. By
  Theorem~\ref{thm:aperyx-rf} with $\lambda = (2, 1)$ and $\varepsilon = 1$,
  the numbers $B (n)$ are the diagonal coefficients of the rational function
  \begin{equation}
    \frac{1}{(1 - x_1 - x_2) (1 - x_3) - x_1 x_2 x_3} = \sum_{\boldsymbol{n} \in
    \mathbb{Z}_{\geq 0}^3} B (\boldsymbol{n}) \boldsymbol{x}^{\boldsymbol{n}}
    . \label{eq:aperyB-rf}
  \end{equation}
  In addition to the binomial sum for $B ( \boldsymbol{n})$ provided by
  Theorem~\ref{thm:aperyx-rf}, MacMahon's Master Theorem~\ref{thm:macmahon}
  shows that $B (n_1, n_2, n_3)$ is the coefficient of $x_1^{n_1} x_2^{n_2}
  x_3^{n_3}$ in the product $(x_1 + x_2 + x_3)^{n_1} (x_1 + x_2)^{n_2} (x_2 +
  x_3)^{n_3}$. An application of Theorem~\ref{thm:aperyx-sc} shows that, for
  $\boldsymbol{n} \in \mathbb{Z}^3$ and integers $r \geq 1$, the
  supercongruences
  \begin{equation}
    B (p^r \boldsymbol{n}) \equiv B (p^{r - 1} \boldsymbol{n}) \pmod{p^{3 r}} \label{eq:aperyB-scr}
  \end{equation}
  hold for all primes $p \geq 5$. In the diagonal case $n_1 = n_2 = n_3$,
  this result was first proved by Coster \cite{coster-sc}.
  
  Proceeding as in Remark~\ref{rk:apery-shifted}, and using the curious
  identity
  \begin{equation}
    \sum_{k = 0}^n \binom{n}{k}^2 \binom{n + k}{k} = \sum_{k = 0}^n ( - 1)^{n
    + k} \binom{n}{k} \binom{n + k}{k}^2,
  \end{equation}
  we find that $B (- n) = ( - 1)^{n - 1} B (n - 1)$ for $n > 0$. Consequently,
  \eqref{eq:aperyB-scr} implies the shifted supercongruences $B (p^r m - 1)
  \equiv B (p^{r - 1} m - 1)$, which hold modulo $p^{3 r}$ for all primes $p
  \geq 5$ and were first proved in \cite{beukers-apery85}, along with
  \eqref{eq:apery-sc1}. We observe that, among the known Ap\'ery-like
  numbers, the sequence $B ( n)$ and the Ap\'ery numbers \eqref{eq:apery}
  are the only ones to satisfy shifted supercongruences of the form
  \eqref{eq:apery-sc1} in addition to the supercongruences of the form
  \eqref{eq:apery-sc}.
\end{example}

\begin{example}
  As a consequence of Theorem~\ref{thm:aperyx-rf} with $\lambda = (3, 1)$ and
  $\varepsilon = 1$, the numbers
  \begin{equation*}
    C (n) = \sum_{k = 0}^n \binom{n}{k}^2 \binom{n + k}{k} \binom{n + 2 k}{k}
  \end{equation*}
  are the diagonal coefficients of the rational function $1 / ((1 - x_1 - x_2
  - x_3) (1 - x_4) - x_1 x_2 x_3 x_4)$. By Theorem~\ref{thm:aperyx-sc}, it
  follows that $C (p^r n) \equiv C (p^{r - 1} n)$ modulo $p^{2 r}$, for all
  primes $p$. We note that this congruence does not, in general, hold modulo a
  larger power of $p$, as is illustrated by $C (5) = 4, 009, 657 \nequiv 7 = C
  (1)$ modulo $5^3$. This demonstrates that in
  Theorem~\ref{thm:aperyx-sc}\ref{thm:aperyx-sc2} the modulus $p^{2 r}$ of the
  congruences cannot, in general, be replaced with $p^{3 r}$, even for $p
  \geq 5$.
\end{example}

\begin{example}
  Next, we consider the sequences
  \begin{equation}
    Y_d (n) = \sum_{k = 0}^n \binom{n}{k}^d . \label{eq:franelx}
  \end{equation}
  The numbers $Y_3 (n)$ satisfy the recurrence \eqref{eq:rec2-abc} with $(a,
  b, c) = (7, 2, - 8)$ and are known as {\emph{Franel numbers}}
  \cite{franel94}, while the numbers $Y_4 (n)$, corresponding to $(a, b, c,
  d) = (6, 2, - 64, 4)$ in \eqref{eq:rec3-abcd}, are sometimes referred to as
  {\emph{Yang--Zudilin numbers}} \cite{ccs-apery}. It follows from
  Theorem~\ref{thm:aperyx-rf} with $\lambda = (1, 1, \ldots, 1)$ and
  $\varepsilon = 1$, that
  \begin{equation}
    \frac{1}{(1 - x_1) (1 - x_2) \cdots (1 - x_d) - x_1 x_2 \cdots x_d} =
    \sum_{\boldsymbol{n} \in \mathbb{Z}_{\geq 0}^d} Y_d (\boldsymbol{n})
    \boldsymbol{x}^{\boldsymbol{n}}, \label{eq:franelx-rf}
  \end{equation}
  where
  \begin{equation}
    Y_d (\boldsymbol{n}) = \sum_{k \geq 0} \binom{n_1}{k} \binom{n_2}{k}
    \cdots \binom{n_d}{k} . \label{eq:franelxx}
  \end{equation}
  It is proved in \cite{ccs-apery} that $Y_d (p n) \equiv Y_d (n)$ modulo
  $p^3$ for primes $p \geq 5$ if $d \geq 2$. These congruences are
  generalized to the multivariate setting by Theorem~\ref{thm:aperyx-sc},
  which shows that, if $d \geq 2$, then, for $\boldsymbol{n} \in
  \mathbb{Z}_{\geq 0}^d$ and integers $r \geq 1$,
  \begin{equation}
    Y_d (p^r \boldsymbol{n}) \equiv Y_d (p^{r - 1} \boldsymbol{n}) \pmod{p^{3 r}} \label{eq:franelx-scr}
  \end{equation}
  for primes $p \geq 5$. Note that
  \begin{equation*}
    Y_2 (\boldsymbol{n}) = \sum_{k \in \mathbb{Z}} \binom{n_1}{k}
     \binom{n_2}{k} = \binom{n_1 + n_2}{n_1} .
  \end{equation*}
  Hence, congruence \eqref{eq:franelx-scr} includes, in particular, the
  appealing binomial congruence
  \begin{equation*}
    \binom{p a}{p b} \equiv \binom{a}{b} \pmod{p^3},
  \end{equation*}
  which is attributed to W.~Ljunggren \cite{granville-bin97} and which
  generalizes the classical congruences by C.~Babbage, J.~Wolstenholme and
  J.~W.~L.~Glaisher. It is further refined by E.~Jacobsthal's binomial
  congruence, which we review in Lemma~\ref{lem:jacobsthal} and which the
  proof of Theorem~\ref{thm:aperyx-sc} crucially depends on.
\end{example}

Let us conclude this section with two conjectural examples, which suggest that
our results are not an isolated phenomenon.

\begin{example}
  As noted in the introduction for the Ap\'ery numbers, there is no unique
  rational function of which a given sequence is the diagonal. For instance,
  the Franel numbers $Y_3 (n)$ are also the diagonal coefficients of the
  rational function
  \begin{equation}
    \frac{1}{1 - (x_1 + x_2 + x_3) + 4 x_1 x_2 x_3} . \label{eq:AG3}
  \end{equation}
  A rational function $F (\boldsymbol{x})$ is said to be {\emph{positive}} if
  its Taylor coefficients \eqref{eq:taylor} are all positive. The
  Askey--Gasper rational function \eqref{eq:AG3}, whose positivity is proved
  in \cite{askey-pos77} and \cite{zb-pos-el83}, is an interesting instance
  of a rational function on the boundary of positivity (if the $4$ is replaced
  by $4 + \varepsilon$, for any $\varepsilon > 0$, then the resulting rational
  function is not positive). The present work was, in part, motivated by the
  observation \cite{sz-pos} that for several of the rational functions,
  which have been shown or conjectured to be on the boundary of positivity,
  the diagonal coefficients are arithmetically interesting sequences with
  links to modular forms. Note that the Askey--Gasper rational function
  \eqref{eq:AG3} corresponds to the choice $\lambda = ( 3)$ and $\alpha = - 4$
  in Theorem~\ref{thm:aperyx-rf}, which makes its Taylor coefficients $G (
  \boldsymbol{n}) = A_{( 3), - 4} ( \boldsymbol{n})$ explicit. We also note that
  an application of MacMahon's Master Theorem~\ref{thm:macmahon} shows that $G
  (n_1, n_2, n_3)$ is the coefficient of $x_1^{n_1} x_2^{n_2} x_3^{n_3}$ in
  the product $(x_1 - x_2 - x_3)^{n_1} (x_2 - x_1 - x_3)^{n_2} (x_3 - x_1 -
  x_2)^{n_3}$. Although it is unclear how one might adjust the proof of
  Theorem~\ref{thm:aperyx-sc}, numerical evidence suggests that the
  coefficients $G (\boldsymbol{n})$ satisfy supercongruences modulo $p^{3 r}$ as
  well.
\end{example}

\begin{conjecture}
  The coefficients $G (\boldsymbol{n})$ of the rational function \eqref{eq:AG3}
  satisfy, for primes $p \geq 5$ and integers $r \geq 1$,
  \begin{equation*}
    G (p^r \boldsymbol{n}) \equiv G (p^{r - 1} \boldsymbol{n}) \pmod{p^{3 r}} .
  \end{equation*}
\end{conjecture}

\begin{example}
  \label{eg:az}Remarkably, the previous example has a four-variable analog,
  which involves the Almkvist--Zudilin numbers $Z ( n)$, introduced in
  \eqref{eq:az}. Namely, the numbers $Z ( n)$ are the diagonal coefficients of
  the unexpectedly simple rational function
  \begin{equation}
    \frac{1}{1 - ( x_1 + x_2 + x_3 + x_4) + 27 x_1 x_2 x_3 x_4}
    \label{eq:az-rat},
  \end{equation}
  as can be deduced from Theorem~\ref{thm:aperyx-rf} with $\lambda = ( 4)$ and
  $\alpha = - 27$. Again, numerical evidence suggests that the coefficients $Z
  (\boldsymbol{n})$ of \eqref{eq:az-rat} satisfy supercongruences modulo $p^{3
  r}$. This is particularly interesting, since even the univariate congruences
  \eqref{eq:az-sc} are conjectural at this time.
\end{example}

\begin{conjecture}
  The coefficients $Z (\boldsymbol{n})$ of the rational function
  \eqref{eq:az-rat} satisfy, for primes $p \geq 5$ and integers $r
  \geq 1$,
  \begin{equation*}
    Z (p^r \boldsymbol{n}) \equiv Z (p^{r - 1} \boldsymbol{n}) \pmod{p^{3 r}} .
  \end{equation*}
\end{conjecture}

\begin{remark}
  The rational functions \eqref{eq:AG3} and \eqref{eq:az-rat} involved in the
  previous examples make it natural to wonder whether supercongruences might
  similarly exist for the family of rational functions given by
  \begin{equation*}
    \frac{1}{1 - ( x_1 + x_2 + \ldots + x_d) + ( d - 1)^{d - 1} x_1 x_2
     \cdots x_d} .
  \end{equation*}
  This does not, however, appear to be the case for $d \geq 5$. In fact,
  no value $b \neq 0$ in
  \begin{equation*}
    \frac{1}{1 - ( x_1 + x_2 + \ldots + x_d) + b x_1 x_2 \cdots x_d}
  \end{equation*}
  appears to give rise to supercongruences (by computing coefficients, we have
  ruled out supercongruences modulo $p^{2 r}$ for integers $| b | < 100, 000$
  and $d \leq 25$).
\end{remark}

\section{The Taylor coefficients}\label{sec:coeffs}

This section is devoted to proving Theorem~\ref{thm:aperyx-rf}. Before we give
a general proof, we offer an alternative approach based on MacMahon's Master
Theorem, to which we refer at several occasions in this note and which offers
additional insight into the Taylor coefficients by expressing them as
coefficients of certain polynomials (see also Remark~\ref{rk:dwork}). This
approach, which we apply here to prove formula \eqref{eq:apery4x}, is based on
the following result of P.~MacMahon \cite{macmahon-comb}, coined by himself
``a master theorem in the Theory of Permutations''. Here,
$[\boldsymbol{x}^{\boldsymbol{m}}]$ denotes the coefficient of $x_1^{m_1} \cdots
x_n^{m_n}$ in the expansion of what follows.

\begin{theorem}
  \label{thm:macmahon}For $\boldsymbol{x}= (x_1, \ldots, x_n)$, matrices $A \in
  \mathbb{C}^{n \times n}$ and $\boldsymbol{m}= (m_1, \ldots, m_n) \in
  \mathbb{Z}_{\geq 0}^n$,
  \begin{equation*}
    [\boldsymbol{x}^{\boldsymbol{m}}]  \prod_{i = 1}^n \left( \sum_{j = 1}^n
     A_{i, j} x_j  \right)^{m_i} = [\boldsymbol{x}^{\boldsymbol{m}}] 
     \frac{1}{\det ( I_n - AX)},
  \end{equation*}
  where $X$ is the diagonal $n \times n$ matrix with entries $x_1, \ldots,
  x_n$.
\end{theorem}

\begin{proof}[Proof of formula \eqref{eq:apery4x}]We note that
  \begin{equation*}
    \frac{1}{(1 - x_1 - x_2) (1 - x_3 - x_4) - x_1 x_2 x_3 x_4} =
     \frac{1}{\det ( I_4 - M X)},
  \end{equation*}
  where $M$ and $X$ are the matrices
  \begin{equation*}
    M = \left(\begin{array}{cccc}
       1 & 1 & 1 & 0\\
       1 & 1 & 0 & 0\\
       0 & 0 & 1 & 1\\
       0 & 1 & 1 & 1
     \end{array}\right), \hspace{1em} X = \left(\begin{array}{cccc}
       x_1 &  &  & \\
       & x_2 &  & \\
       &  & x_3 & \\
       &  &  & x_4
     \end{array}\right) .
  \end{equation*}
  An application of MacMahon's Master Theorem~\ref{thm:macmahon} therefore
  shows that the coefficients $A (\boldsymbol{n})$, with $\boldsymbol{n}= (n_1,
  n_2, n_3, n_4)$, are given by
  \begin{equation*}
    A (\boldsymbol{n}) = [\boldsymbol{x}^{\boldsymbol{n}}] (x_1 + x_2 + x_3)^{n_1}
     (x_1 + x_2)^{n_2} (x_3 + x_4)^{n_3} (x_2 + x_3 + x_4)^{n_4} .
  \end{equation*}
  In order to extract the requisite coefficient, we expand the right-hand side
  as
  \begin{eqnarray*}
    &  & (x_1 + x_2 + x_3)^{n_1} (x_1 + x_2)^{n_2} (x_3 + x_4)^{n_3} (x_2 +
    x_3 + x_4)^{n_4}\\
    & = & \sum_{k_1, k_4} \binom{n_1}{k_1} \binom{n_4}{k_4} x_2^{n_4 - k_4}
    x_3^{n_1 - k_1} (x_1 + x_2)^{k_1 + n_2} (x_3 + x_4)^{n_3 + k_4}\\
    & = & \sum_{k_1, k_2, k_3, k_4} \binom{n_1}{k_1} \binom{n_4}{k_4}
    \binom{k_1 + n_2}{k_2} \binom{n_3 + k_4}{k_3} x_1^{k_1 + n_2 - k_2}
    x_2^{n_4 - k_4 + k_2} x_3^{n_1 - k_1 + k_3} x_4^{n_3 + k_4 - k_3} .
  \end{eqnarray*}
  The summand contributes to $x_1^{n_1} x_2^{n_2} x_3^{n_3} x_4^{n_4}$ if and
  only if $n_i - k_i = n_j - k_j$ for all $i, j = 1, \ldots, 4$. Writing $k =
  n_i - k_i$ for the common value, we obtain
  \begin{equation*}
    A (n_1, n_2, n_3, n_4) = \sum_{k \in \mathbb{Z}} \binom{n_1}{k}
     \binom{n_4}{k} \binom{n_1 - k + n_2}{n_2 - k} \binom{n_3 + n_4 - k}{n_3 -
     k},
  \end{equation*}
  which is equivalent to the claimed \eqref{eq:apery4x}.
\end{proof}

\begin{proof}[Proof of Theorem \ref{thm:aperyx-rf}]
  Recall the elementary formula
  \begin{equation*}
    \frac{1}{(1 - x)^{k + 1}} = \sum_{n \geq 0} \binom{n + k}{k} x^n,
  \end{equation*}
  for integers $k \geq 0$. Combined with an application of the
  multinomial theorem, it implies that
  \begin{equation*}
    \frac{1}{(1 - x_1 - \ldots - x_{\rho})^{k + 1}} = \sum_{n_1 \geq 0}
     \cdots \sum_{n_{\rho} \geq 0} \binom{n_1 + \ldots + n_{\rho} +
     k}{n_1, \ldots, n_{\rho}, k} x_1^{n_1} \cdots x_{\rho}^{n_{\rho}},
  \end{equation*}
  and hence
  \begin{equation}
    \frac{(x_1 \cdots x_{\rho})^k}{(1 - x_1 - \ldots - x_{\rho})^{k + 1}} =
    \sum_{n_1 \geq 0} \cdots \sum_{n_{\rho} \geq 0} \binom{n_1 +
    \ldots + n_{\rho} - (\rho - 1) k}{n_1 - k, \ldots, n_{\rho} - k, k}
    x_1^{n_1} \cdots x_{\rho}^{n_{\rho}} . \label{eq:binomexp}
  \end{equation}
  Here, we used that the multinomial coefficient vanishes if $k > \min (n_1,
  \ldots, n_{\rho})$. Geometrically expanding the left-hand side of
  \eqref{eq:aperyx-rf}, we find that
  \begin{equation*}
    \left( \prod_{j = 1}^{\ell} \left[ 1 - \sum_{r = 1}^{\lambda_j} x_{s (j)
     + r} \right] - \alpha x_1 x_2 \cdots x_d \right)^{- 1} = \sum_{k
     \geq 0} \alpha^k \prod_{j = 1}^{\ell} \frac{(x_{s (j) + 1} \cdots
     x_{s (j) + \lambda_j})^k}{\left[ 1 - \sum_{r = 1}^{\lambda_j} x_{s (j) +
     r} \right]^{k + 1}},
  \end{equation*}
  which we further expand using \eqref{eq:binomexp} to get
  \begin{equation*}
    \sum_{k \geq 0} \alpha^k \sum_{\boldsymbol{n} \in
     \mathbb{Z}_{\geq 0}^d} \boldsymbol{x}^{\boldsymbol{n}} \prod_{j =
     1}^{\ell} \binom{n_{s (j) + 1} + \ldots + n_{s (j) + \lambda_j} -
     (\lambda_j - 1) k}{n_{s (j) + 1} - k, \ldots, n_{s (j) + \lambda_j} - k,
     k} = \sum_{\boldsymbol{n} \in \mathbb{Z}_{\geq 0}^d} A_{\lambda,
     \alpha} (\boldsymbol{n}) \boldsymbol{x}^{\boldsymbol{n}},
  \end{equation*}
  with $A_{\lambda, \alpha} (\boldsymbol{n})$ as in \eqref{eq:Alambda}.
\end{proof}

\section{The supercongruences}\label{sec:sc}

Our proof of Theorem~\ref{thm:aperyx-sc}, which generalizes the
supercongruence in Theorem~\ref{thm:apery4-sc}, builds upon the respective
proofs in \cite{gessel-super} and \cite{beukers-apery85}.

We need a number of lemmas in preparation. To begin with, we prove the
following extension of Jacobsthal's binomial congruence \cite{gessel-euler},
\cite{granville-bin97} to binomial coefficients which are allowed to have
negative entries (see Remark~\ref{rk:apery-shifted}).

\begin{lemma}
  \label{lem:jacobsthal}For all primes $p$ and all integers $a, b$,
  \begin{equation}
    \binom{a p}{b p} / \binom{a}{b} \equiv \varepsilon \pmod{p^q}, \label{eq:jacobsthal}
  \end{equation}
  where $q$ is the power of $p$ dividing $p^3 a b (a - b) / 12$ and where
  $\varepsilon = 1$, unless $p = 2$ and $(a, b) \equiv (0, 1)$ modulo $2$ in
  which case $\varepsilon = - 1$.
\end{lemma}

\begin{proof}
  Congruence \eqref{eq:jacobsthal}, for nonnegative $a, b$, is proved in
  \cite{gessel-euler} (alternatively, a proof for $p \geq 5$ is given
  in \cite{granville-bin97}). We therefore only indicate how to extend
  \eqref{eq:jacobsthal} to negative values of $a$ or $b$. Note that, for all
  $a, b \in \mathbb{Z}$ with $b \neq 0$,
  \begin{equation*}
    \binom{a}{b} = \frac{a}{b} \binom{a - 1}{b - 1},
  \end{equation*}
  and hence
  \begin{equation*}
    \binom{a p}{b p} / \binom{a}{b} = \binom{a p - 1}{b p - 1} / \binom{a -
     1}{b - 1} .
  \end{equation*}
  We claim that the extension of \eqref{eq:jacobsthal} to the case $a < 0$ and
  $b < 0$ therefore follows from
  \begin{equation}
    \binom{a}{b} = \binom{- b - 1}{- a - 1} (- 1)^{a - b} \operatorname{sgn} (a - b),
    \label{eq:binom-neg2}
  \end{equation}
  where $\operatorname{sgn}$ is defined as in Remark~\ref{rk:apery-shifted}. This is
  clear for $p \geq 3$. Write $\varepsilon (a, b) = - 1$ if $(a, b)
  \equiv (0, 1)$ modulo $2$ and $\varepsilon (a, b) = 1$ otherwise. It is
  straightforward to check that
  \begin{equation*}
    (- 1)^{a - b} \varepsilon (- b, - a) = \varepsilon (a, b),
  \end{equation*}
  which shows the case $p = 2$.
  
  Similarly, if $a < 0$ and $b > 0$, then we may apply
  \begin{equation*}
    \binom{a}{b} = \binom{b - a - 1}{- a - 1} (- 1)^{b + 1} \operatorname{sgn} (a -
     b) \operatorname{sgn} (- a - 1)
  \end{equation*}
  as well as
  \begin{equation*}
    (- 1)^b \varepsilon (b - a, - a) = \varepsilon (a, b) .
  \end{equation*}
  A derivation of the above binomial identities, which are valid for all $a, b
  \in \mathbb{Z}$, may be found in \cite{sprugnoli-binom}.
\end{proof}

Much simpler and well-known is the following congruence.

\begin{lemma}
  \label{lem:powersummod}Let $p \geq 5$ be a prime, and $\varepsilon \in
  \{ - 1, 1 \}$. Then, for all integers $r \geq 0$,
  \begin{equation}
    \sum_{k = 1, p \nmid k}^{p^r - 1} \frac{\varepsilon^k}{k^2} \equiv 0
    \pmod{p^r} . \label{eq:powersummod}
  \end{equation}
\end{lemma}

\begin{proof}
  Let $\alpha$ be an odd integer, not divisible by $p$, such that $\alpha^2
  \nequiv 1$ modulo $p$ (take, for instance, $\alpha = 3$). Then,
  \begin{equation*}
    \frac{1}{\alpha^2} \sum_{k = 1, p \nmid k}^{p^r - 1}
     \frac{\varepsilon^k}{k^2} = \sum_{k = 1, p \nmid k}^{p^r - 1}
     \frac{\varepsilon^k}{( \alpha k)^2} \equiv \sum_{k = 1, p \nmid k}^{p^r -
     1} \frac{\varepsilon^k}{k^2} \pmod{p^r},
  \end{equation*}
  since the second and third sum run over the same residues modulo $p^r$ (note
  that $\varepsilon^{\alpha k} = \varepsilon^k$ since $\alpha$ is odd). As
  $\alpha^2$ is not divisible by $p$, the congruence \eqref{eq:powersummod}
  follows.
\end{proof}

The next lemmas establish properties of the summands of the numbers
$A_{\lambda, \varepsilon} (\boldsymbol{n})$ as introduced in \eqref{eq:Alambda},
which will be needed in our proof of Theorem~\ref{thm:aperyx-sc}. Throughout
this section, we fix the notation of Theorem~\ref{thm:aperyx-sc}, letting
$\lambda = (\lambda_1, \ldots, \lambda_{\ell}) \in \mathbb{Z}_{> 0}^{\ell}$
with $d = \lambda_1 + \ldots + \lambda_{\ell}$ and setting $s (j) = \lambda_1
+ \ldots + \lambda_{j - 1}$.

\begin{lemma}
  \label{lem:Alambdak}Let $\boldsymbol{n} \in \mathbb{Z}^d$, $k \in
  \mathbb{Z}$, and define
  \begin{equation}
    A_{\lambda} (\boldsymbol{n}; k) = \prod_{j = 1}^{\ell} \binom{n_{s (j) + 1}
    + \ldots +_{} n_{s (j) + \lambda_j} - (\lambda_j - 1) k}{n_{s (j) + 1} -
    k, \ldots, n_{s (j) + \lambda_j} - k, k} . \label{eq:Alambdak}
  \end{equation}
  \begin{enumeratealpha}
    \item \label{n:Ac2}If $\ell \geq 2$, then, for all primes $p$ and
    integers $r \geq 1$,
    \begin{equation}
      A_{\lambda} (p^r \boldsymbol{n}; p k) \equiv A_{\lambda} (p^{r - 1}
      \boldsymbol{n}; k) \pmod{p^{2 r}} . \label{eq:Akind2}
    \end{equation}
    \item \label{n:Ac3}If $\ell \geq 2$ and $\max (\lambda_1, \ldots,
    \lambda_{\ell}) \leq 2$, then, for primes $p \geq 5$ and
    integers $r \geq 1$,
    \begin{equation}
      A_{\lambda} (p^r \boldsymbol{n}; p k) \equiv A_{\lambda} (p^{r - 1}
      \boldsymbol{n}; k) \pmod{p^{3 r}} . \label{eq:Akind}
    \end{equation}
  \end{enumeratealpha}
\end{lemma}

\begin{proof}
  We show \eqref{eq:Akind2} and \eqref{eq:Akind} by proving that for integers
  $r, s \geq 1$ and $k$ such that $p \nmid k$,
  \begin{equation}
    A_{\lambda} (p^r \boldsymbol{n}; p^s k) \equiv A_{\lambda} (p^{r - 1}
    \boldsymbol{n}; p^{s - 1} k) \pmod{p^{\alpha r}},
    \label{eq:Akind1}
  \end{equation}
  where $\alpha = 2$ or $\alpha = 3$ depending on whether $\max (\lambda_1,
  \ldots, \lambda_{\ell}) \leq 2$.
  
  Let us first consider the case $\ell \geq 2$ and $\max (\lambda_1,
  \ldots, \lambda_{\ell}) \leq 2$. Then each factor of
  \eqref{eq:Alambdak} is a single binomial, if $\lambda_j = 1$, or of the form
  \begin{equation*}
    \binom{m_1}{k} \binom{m_1 + m_2 - k}{m_1},
  \end{equation*}
  if $\lambda_j = 2$. Let $p$ be a prime such that $p \geq 5$. It follows
  from Jacobsthal's congruence \eqref{eq:jacobsthal} that
  \begin{equation*}
    \binom{p^r m_1}{p^s k} / \binom{p^{r - 1} m_1}{p^{s - 1} k} \equiv 1
     \pmod{p^{r + s + \min (r, s)}}
  \end{equation*}
  as well as
  \begin{equation*}
    \binom{p^r (m_1 + m_2) - p^s k}{p^r m_1} / \binom{p^{r - 1} (m_1 + m_2) -
     p^{s - 1} k}{p^{r - 1} m_1} \equiv 1 \pmod{p^{r + 2
     \min (r, s)}} .
  \end{equation*}
  Consequently,
  \begin{equation}
    A_{\lambda} (p^r \boldsymbol{n}; p^s k) = c A_{\lambda} (p^{r - 1}
    \boldsymbol{n}; p^{s - 1} k) \label{eq:Akind0}
  \end{equation}
  with $c \equiv 1$ modulo $p^{r + 2 \min (r, s)}$. If $s \geq r$, this
  proves congruence \eqref{eq:Akind1} with $\alpha = 3$. On the other hand,
  suppose $s \leq r$. Since $p \nmid k$, we have
  \begin{equation*}
    \binom{p^r n}{p^s k} = p^{r - s}  \frac{n}{k} \binom{p^r n - 1}{p^s k -
     1} \equiv 0 \pmod{p^{r - s}} .
  \end{equation*}
  Since $\ell \geq 2$, it follows that $p^{2 (r - s)}$ divides
  $A_{\lambda} (p^r \boldsymbol{n}; p^s k)$. Since $(r + 2 s) + 2 (r - s) = 3
  r$, the congruence \eqref{eq:Akind1}, with $\alpha = 3$, now follows from
  \eqref{eq:Akind0}. This shows \ref{n:Ac3}.
  
  Let us now turn to the proof of \ref{n:Ac2}. Assume that $\ell \geq 2$.
  Note that, for any positive integer $\rho$,
  \begin{equation*}
    \binom{m_1 + \ldots + m_{\rho} - (\rho - 1) k}{m_1 - k, \ldots, m_{\rho}
     - k, k} = \binom{m_1}{k} \binom{m_1 + (m_2 - k) + \ldots + (m_{\rho} -
     k)}{m_1, m_2 - k, \ldots, m_{\rho} - k},
  \end{equation*}
  so that, as in the previous case, $p^{\ell (r - s)}$ divides $A_{\lambda}
  (p^r \boldsymbol{n}; p^s k)$ if $r \geq s$.
  
  Initially, assume that $p \geq 3$. By further unravelling the
  multinomial coefficient as a product of binomial coefficients and applying
  Jacobsthal's congruence \eqref{eq:jacobsthal} as above, we find that
  \begin{equation*}
    A_{\lambda} (p^r \boldsymbol{n}; p^s k) = c A_{\lambda} (p^{r - 1}
     \boldsymbol{n}; p^{s - 1} k)
  \end{equation*}
  with $c \equiv 1$ modulo $p^{3 \min (r, s) - \delta}$ and $\delta = 0$, if
  $p \geq 5$, and $\delta = 1$, if $p = 3$. In light of $p^{2 (r - s)}$
  dividing $A_{\lambda} (p^r \boldsymbol{n}; p^s k)$ if $r \geq s$, we
  conclude congruence \eqref{eq:Akind1} with $\alpha = 2$.
  
  Now, consider $p = 2$. If $r \geq 2$ and $s \geq 2$, then the sign
  $\varepsilon$ in Jacobsthal's congruence \eqref{eq:jacobsthal} is always $+
  1$ when applying the above approach, and we again find that
  \eqref{eq:Akind1} holds with $\alpha = 2$. On the other hand, if $r = 1$,
  then it suffices to use the (combinatorial) congruence
  \begin{equation*}
    \binom{p a}{p b} \equiv \binom{a}{b} \pmod{p^2},
  \end{equation*}
  which holds for all primes $p$. It remains to consider the case $r \geq
  2$ and $s = 1$. Applying the approach employed for $p \geq 3$, we find
  that
  \begin{equation}
    A_{\lambda} (p^r \boldsymbol{n}; p^s k) = c A_{\lambda} (p^{r - 1}
    \boldsymbol{n}; p^{s - 1} k), \label{eq:Akind02}
  \end{equation}
  where $c \equiv \pm 1$ modulo $p^{3 \min (r, s) - 2} = 2$. If $\max
  (\lambda_1, \ldots, \lambda_{\ell}) \leq 2$, then we, in fact, have $c
  \equiv (- 1)^{\ell}$ modulo $p^{r + 2 \min (r, s) - 2} = 2^r$. Since
  $A_{\lambda} (p^r \boldsymbol{n}; p^s k)$ is divisible by $p^{\ell (r - 1)}$,
  congruence \eqref{eq:Akind1} trivially holds with $\alpha = 2$ if $\ell
  \geq 3$. Hence, we may assume that $\ell = 2$. If $\max (\lambda_1,
  \lambda_2) \leq 2$, then $c \equiv 1$ modulo $2^r$ in
  \eqref{eq:Akind02} and, since both sides of \eqref{eq:Akind02} are divisible
  by $2^{2 r - 2}$, congruence \eqref{eq:Akind1} with $\alpha = 2$ again
  follows. Finally, suppose that there is $j$ such that $\lambda_j \geq
  3$. Then the factor corresponding to $j$ in \eqref{eq:Alambdak} is of the
  form
  \begin{equation*}
    \binom{m_1}{k} \binom{m_1 + m_2 - k}{m_1} \binom{m_1 + m_2 + m_3 - 2
     k}{m_3 - k} \binom{m_1 + \ldots + m_{\rho} - (\rho - 1) k}{m_1 + m_2 +
     m_3 - 2 k, m_4 - k, \ldots} .
  \end{equation*}
  Note that, for even $m_1, m_2, m_3$ and odd $k$, the third binomial in this
  product is even. Hence, $A_{\lambda} (p^r \boldsymbol{n}; p^s k)$ is divisible
  by $2^{2 (r - 1) + 1} = 2^{2 r - 1}$. In light of \eqref{eq:Akind02}, this
  proves congruence \eqref{eq:Akind1} with $\alpha = 2$.
\end{proof}

The next congruence, with $k \geq 0$, has been used in
\cite{beukers-apery85}. For our present purpose, we extend it to the case of
negative $k$.

\begin{lemma}
  \label{lem:binomind}For primes $p$, integers $m, k$ and integers $r
  \geq 1$,
  \begin{equation}
    \binom{p^r m - 1}{k} (- 1)^k \equiv \binom{p^{r - 1} m - 1}{[k / p]} (-
    1)^{[k / p]} \pmod{p^r} . \label{eq:binomind}
  \end{equation}
\end{lemma}

\begin{proof}
  First, assume that $k \geq 0$. Following \cite[Lemma
  2]{beukers-apery85}, we split the defining product of the binomial
  coefficient, according to whether the index is divisible by $p$ or not, to
  obtain
  \begin{eqnarray*}
    \binom{p^r m - 1}{k} & = & \prod_{j = 1}^k \frac{p^r m - j}{j}\\
    & = & \prod_{j = 1, p \nmid j}^k \frac{p^r m - j}{j} \prod_{\lambda =
    1}^{[k / p]} \frac{p^{r - 1} m - \lambda}{\lambda}\\
    & = & \binom{p^{r - 1} m - 1}{[k / p]} \prod_{j = 1, p \nmid j}^k
    \frac{p^r m - j}{j} .
  \end{eqnarray*}
  Congruence \eqref{eq:binomind}, with $k \geq 0$, follows upon reducing
  modulo $p^r$.
  
  On the other hand, assume $k < 0$. Since \eqref{eq:binomind} is trivial if
  $m > 0$, we let $m \leq 0$. We use the basic symmetry relation
  \begin{equation*}
    \binom{p^r m - 1}{k} = \binom{p^r m - 1}{p^r m - k - 1}
  \end{equation*}
  and note that, since $k < 0$, the binomials are zero unless $p^r m - k - 1
  \geq 0$. Observe that, for all integers $k, m$,
  \begin{equation}
    [(p^r m - k - 1) / p] = p^{r - 1} m + [- (k + 1) / p] = p^{r - 1} m - [k /
    p] - 1. \label{eq:brm}
  \end{equation}
  Thus, assuming $p^r m - 1 - k \geq 0$, we may apply \eqref{eq:binomind}
  to find
  \begin{eqnarray*}
    \binom{p^r m - 1}{k} (- 1)^k & = & \binom{p^r m - 1}{p^r m - k - 1} (-
    1)^k\\
    & \equiv & \binom{p^{r - 1} m - 1}{p^{r - 1} m - [k / p] - 1} (- 1)^{[k /
    p]} (- 1)^{p^r m + p^{r - 1} m}\\
    & = & \binom{p^{r - 1} m - 1}{[k / p]} (- 1)^{[k / p]} (- 1)^{p^r m +
    p^{r - 1} m} \pmod{p^r} .
  \end{eqnarray*}
  It only remains to note that $p^r m + p^{r - 1} m = p^{r - 1} (p + 1) m$ is
  even unless $p = 2$ and $r = 1$. Hence, in all cases, $(- 1)^{p^r m + p^{r -
  1} m} \equiv 1$ modulo $p^r$.
\end{proof}

\begin{lemma}
  \label{lem:binomind2}For primes $p$, integers $m_1, m_2, k$ and integers $r
  \geq 1$,
  \begin{equation*}
    \binom{p^r m_1 + p^r m_2 - k - 1}{p^r m_1} \equiv \binom{p^{r - 1} m_1 +
     p^{r - 1} m_2 - [k / p] - 1}{p^{r - 1} m_1} \pmod{p^r}
     .
  \end{equation*}
\end{lemma}

\begin{proof}
  By an application of \eqref{eq:binom-neg},
  \begin{equation*}
    \binom{m_1 + m_2 - k - 1}{m_1} = \operatorname{sgn} (m_2 - k - 1) (- 1)^{m_2 - k
     - 1} \binom{- m_1 - 1}{m_2 - k - 1} .
  \end{equation*}
  Since, for all $a \in \mathbb{Z}$, $\operatorname{sgn} (a) = \operatorname{sgn} ([a / p])$,
  the claimed congruence therefore follows from \eqref{eq:brm} and
  Lemma~\ref{lem:binomind}.
\end{proof}

The following generalizes \cite[Lemma 3]{beukers-apery85} to our needs.

\begin{lemma}
  \label{lem:sumCind}Let $p$ be a prime and $\boldsymbol{n} \in \mathbb{Z}^d$.
  \begin{itemize}
    \item Let $a_k \in \mathbb{Z}_p$, with $k \in \mathbb{Z}$, be such that,
    for all $l, s \in \mathbb{Z}$ with $s \geq 0$,
    \begin{equation*}
      \sum_{[k / p^s] = l} a_k \equiv 0 \pmod{p^s} .
    \end{equation*}
    \item Let $C (\boldsymbol{n}; k)$ be such that, for all $k, r \in
    \mathbb{Z}$ with $r \geq 0$,
    \begin{equation}
      C (p^r \boldsymbol{n}; k) \equiv C (p^{r - 1} \boldsymbol{n}; [k / p])
      \pmod{p^r} . \label{eq:C-assum}
    \end{equation}
  \end{itemize}
  Then, for all $r, l \in \mathbb{Z}$ with $r \geq 0$,
  \begin{equation}
    \sum_{[k / p^r] = l} a_k C (p^r \boldsymbol{n}; k) \equiv 0 \pmod{p^r} . \label{eq:sumCind}
  \end{equation}
\end{lemma}

\begin{proof}
  The claim is trivial for $r = 0$. Fix $r > 0$ and assume, for the purpose of
  induction on $r$, that the congruence \eqref{eq:sumCind} holds for the
  exponent $r - 1$ in place of $r$. By the assumption \eqref{eq:C-assum} on $C
  (\boldsymbol{n}; k)$, we have that, modulo $p^r$,
  \begin{eqnarray*}
    \sum_{[k / p^r] = l} a_k C (p^r \boldsymbol{n}; k) & \equiv & \sum_{[k /
    p^r] = l} a_k C (p^{r - 1} \boldsymbol{n}; [k / p])\\
    & = & \sum_{[m / p^{r - 1}] = l} \left( \sum_{[k / p] = m} a_k \right) C
    (p^{r - 1} \boldsymbol{n}; m)\\
    & = & p \sum_{[m / p^{r - 1}] = l} b_m C (p^{r - 1} \boldsymbol{n}; m),
  \end{eqnarray*}
  where $b_m$ is the sequence
  \begin{equation*}
    b_m = \frac{1}{p} \sum_{[k / p] = m} a_k .
  \end{equation*}
  We note that, for all $s, l \in \mathbb{Z}$ with $s \geq 0$,
  \begin{equation*}
    \sum_{[m / p^s] = l} b_m = \frac{1}{p} \sum_{[m / p^s] = l} \sum_{[k / p]
     = m} a_k = \frac{1}{p} \sum_{[k / p^{s + 1}] = m} a_k \equiv 0
     \pmod{p^s},
  \end{equation*}
  so that we may apply our induction hypothesis \eqref{eq:sumCind} with $r -
  1$ to conclude
  \begin{equation*}
    \sum_{[k / p^r] = l} a_k C (p^r \boldsymbol{n}; k) = p \sum_{[m / p^{r -
     1}] = l} b_m C (p^{r - 1} \boldsymbol{n}; m) \equiv 0 \pmod{p^r} .
  \end{equation*}
  The claim therefore follows by induction.
\end{proof}

We are now in a comfortable position to prove Theorem \ref{thm:aperyx-sc}.

\begin{proof}[Proof of Theorem \ref{thm:aperyx-sc}]
  In terms of the numbers $A_{\lambda,
  \varepsilon} (\boldsymbol{n}; k)$, defined in \eqref{eq:Alambdak}, we have
  \begin{equation*}
    A_{\lambda, \varepsilon} (\boldsymbol{n}) = \sum_{k \geq 0}
     \varepsilon^k A_{\lambda} (\boldsymbol{n}; k) = \sum_{s \geq 0} G_s
     (\boldsymbol{n}),
  \end{equation*}
  where
  \begin{equation*}
    G_s (\boldsymbol{n}) = \sum_{p \nmid k} \varepsilon^{p^s k} A_{\lambda}
     (\boldsymbol{n}; p^s k) .
  \end{equation*}
  Suppose that $\ell \geq 2$. Further, suppose that $p \geq 3$, or
  that $p = 2$ and $\varepsilon = 1$. Then $\varepsilon^{p^s k} =
  \varepsilon^{p^{s - 1} k}$, and it follows from Lemma~\ref{lem:Alambdak}
  that, for $s \geq 1$,
  \begin{equation*}
    G_s (p^r \boldsymbol{n}) \equiv G_{s - 1} (p^{r - 1} \boldsymbol{n})
     \pmod{p^{2 r}} .
  \end{equation*}
  In order to prove that $A_{\lambda, \varepsilon} (p^r \boldsymbol{n}) \equiv
  A_{\lambda, \varepsilon} (p^{r - 1} \boldsymbol{n})$ modulo $p^{2 r}$, it
  therefore remains only to show that $G_0 (p^r \boldsymbol{n}) \equiv 0$ modulo
  $p^{2 r}$. This, however, is immediate because, as observed in the proof of
  Lemma~\ref{lem:Alambdak}, $A_{\lambda} (p^r \boldsymbol{n}; k)$, with $p \nmid
  k$, is divisible by $p^{\ell r}$. This proves congruence
  \eqref{eq:aperyx-sc2}.
  
  Now, suppose that $\ell \geq 2$ and $\max (\lambda_1, \ldots,
  \lambda_{\ell}) \leq 2$. Let $p$ be a prime such that $p \geq 5$.
  It again follows from $\varepsilon^{p^s k} = \varepsilon^{p^{s - 1} k}$ and
  Lemma~\ref{lem:Alambdak} that, for $s \geq 1$,
  \begin{equation*}
    G_s (p^r \boldsymbol{n}) \equiv G_{s - 1} (p^{r - 1} \boldsymbol{n})
     \pmod{p^{3 r}} .
  \end{equation*}
  To prove that $A_{\lambda, \varepsilon} (p^r \boldsymbol{n}) \equiv
  A_{\lambda, \varepsilon} (p^{r - 1} \boldsymbol{n})$ modulo $p^{3 r}$, we have
  to show that $G_0 (p^r \boldsymbol{n}) \equiv 0$ modulo $p^{3 r}$. As in the
  previous case, this is trivial if $\ell \geq 3$. We thus assume $\ell =
  2$.
  
  Note that, since $\max (\lambda_1, \ldots, \lambda_{\ell}) \leq 2$,
  each factor of $A_{\lambda} (\boldsymbol{n}; k)$ is of the form
  \begin{equation*}
    \binom{m_1}{k}, \hspace{1em} \text{or} \hspace{1em} \binom{m_1}{k}
     \binom{m_1 + m_2 - k}{m_1} .
  \end{equation*}
  Using the basic identity
  \begin{equation*}
    \binom{m_1}{k} = \frac{m_1}{k} \binom{m_1 - 1}{k - 1},
  \end{equation*}
  it is clear that the numbers
  \begin{equation*}
    B_{\lambda} (\boldsymbol{n}; k) = \frac{k^2}{n_1 n_{1 + \lambda_1}}
     A_{\lambda} (\boldsymbol{n}; k)
  \end{equation*}
  are integers. Moreover, it follows from Lemmas~\ref{lem:binomind} and
  \ref{lem:binomind2}, and the fact that $\ell = 2$, that the integers
  $C_{\lambda} (\boldsymbol{n}; k) = B_{\lambda} (\boldsymbol{n}; k + 1)$ satisfy,
  for all $k, r \in \mathbb{Z}$ with $r \geq 0$,
  \begin{equation*}
    C (p^r \boldsymbol{n}; k) \equiv C (p^{r - 1} \boldsymbol{n}; [k / p])
     \pmod{p^r} .
  \end{equation*}
  If $p \nmid k$ then $[(k - 1) / p] = [k / p]$ so that, in particular,
  \begin{equation*}
    C (p^r \boldsymbol{n}, k - 1) \equiv C (p^r \boldsymbol{n}, [k / p]) \equiv C
     (p^r \boldsymbol{n}; k) \pmod{p^r} .
  \end{equation*}
  By construction,
  \begin{equation*}
    G_0 (p^r \boldsymbol{n}) = p^{2 r} n_1 n_{1 + \lambda_1} \sum_{p \nmid k}
     \frac{\varepsilon^k}{k^2} C (p^r \boldsymbol{n}; k - 1),
  \end{equation*}
  so that, in order to show that $G_0 (p^r \boldsymbol{n}) \equiv 0$ modulo
  $p^{3 r}$, it suffices to prove
  \begin{equation}
    \sum_{p \nmid k} \frac{\varepsilon^k}{k^2} C (p^r \boldsymbol{n}; k) \equiv
    0 \pmod{p^r} . \label{eq:sumC}
  \end{equation}
  Define $a_k = \varepsilon^k / k^2$, if $p \nmid k$, and $a_k = 0$ otherwise.
  Since $p \geq 5$, it follows from Lemma~\ref{lem:powersummod} that, for
  all $l, s \in \mathbb{Z}$ with $s \geq 0$,
  \begin{equation*}
    \sum_{[k / p^s] = l} a_k = \sum_{k = 1, p \nmid k}^{p^s - 1}
     \frac{\varepsilon^{l p^s + k}}{(l p^s + k)^2} \equiv \varepsilon^l
     \sum_{k = 1, p \nmid k}^{p^s - 1} \frac{\varepsilon^k}{k^2} \equiv 0
     \pmod{p^s} .
  \end{equation*}
  Hence, the conditions of Lemma~\ref{lem:sumCind} are met, allowing us to
  conclude that
  \begin{equation*}
    \sum_{p \nmid k} \frac{\varepsilon^k}{k^2} C (p^r \boldsymbol{n}; k) =
     \sum_l \sum_{[k / p^r] = l} a_k C (p^r \boldsymbol{n}; k) \equiv 0
     \pmod{p^r} .
  \end{equation*}
  This shows \eqref{eq:sumC} and completes our proof.
\end{proof}

\begin{acknowledgements}
I wish to thank Robert Osburn and Wadim Zudilin
for interesting and motivating discussions on Ap\'ery-like numbers and
supercongruences. Much appreciated comments on an earlier version of this note
have also been sent by Bruce Berndt and Brundaban Sahu. Moreover, I am very
grateful to the referee for several suggestions that improved this paper.
Finally, I thank the Max-Planck-Institute for Mathematics in Bonn, where most
of this work was completed, for providing wonderful working conditions.
\end{acknowledgements}


\newcommand{\etalchar}[1]{$^{#1}$}

\end{document}